\documentclass[12pt]{amsart}
\usepackage{amssymb}
\usepackage{graphicx}
\usepackage{colordvi}

\numberwithin{equation}{section}
\newtheorem{theorem}{Theorem}[section]
\newtheorem{proposition}[theorem]{Proposition}
\newtheorem{lemma}[theorem]{Lemma}
\newtheorem{corollary}[theorem]{Corollary}
\theoremstyle{remark}
\newtheorem{example}[theorem]{Example}
\newtheorem{remark}[theorem]{Remark}

\newcounter{FNC}[page]
\def\fauxfootnote#1{{\addtocounter{FNC}{2}$^\fnsymbol{FNC}$%
     \let\thefootnote\relax\footnotetext{$^\fnsymbol{FNC}$#1}}}
\newcommand{\defcolor}[1]{\Blue{#1}}
\newcommand{\demph}[1]{\defcolor{{\sl #1}}}

\newcommand{\C}{{\mathbb{C}}}

\newcommand{\Z}{{\mathbb{Z}}}

\newcommand{\calE}{{\mathcal{E}}}
\newcommand{\calF}{{\mathcal{F}}}
\newcommand{\calG}{{\mathcal{G}}}

\newcommand{\calI}{{\mathcal{I}}}
\newcommand{\calO}{{\mathcal{O}}}

\newcommand{\frakC}{{\mathfrak{C}}}
\newcommand{\frakh}{{\mathfrak{h}}}
\newcommand{\frakS}{{\mathfrak{S}}}

\title{Cohomological consequences of the pattern map}

\author{Praise Adeyemo}
\address{Department of Mathematics\\
  University of Ibadan\\
   Ibadan, Oyo, Nigeria}
\email{ph.adeyemo@ui.edu.ng}

\author{Frank Sottile}
\address{Department of Mathematics\\
         Texas A\&M University\\
         College Station\\
         Texas \ 77843\\
         USA}
\email{sottile@math.tamu.edu}
\urladdr{http://www.math.tamu.edu/\~{}sottile}

\thanks{Research of Sottile supported in part by NSF grant DMS-1001615}
\thanks{We thank the Universities of Ibadan and Ilorin, who supported the visit of Sottile in July 2014}

\begin{document}

\begin{abstract}
 Billey and Braden defined maps on flag manifolds that are the geometric counterpart of permutation patterns.
 A section of their pattern map is an embedding of the flag manifold of a Levi subgroup into
 the full flag manifold.
 We give two expressions for the induced map on cohomology.
 One is in terms of generators and the other is in terms of the Schubert basis.
 We show that the coefficients in the second expression are naturally Schubert structure
 constants and therefore positive.
 Similar results hold for $K$-theory, generalizing known formulas in type $A$ for cohomology and
 $K$-theory. 
\end{abstract}

\maketitle
%
\section*{Introduction}
In their study of singularities of Schubert varieties and coefficients of Kazhdan-Lusztig
polynomials~\cite{BB}, Billey and Braden introduced maps of flag
manifolds that are the geometric counterpart of the generalized permutation patterns of
Billey and Postnikov~\cite{BiPo}. 
We study sections of the Billey-Braden pattern map.
For the type $A$ flag manifold, such sections led to 
formulas for certain specializations of Schubert~\cite{BS98} and Grothendieck polynomials~\cite{LRS}.  
In both cases, this gave new expressions for Schubert class representatives as explicit sums of
monomials~\cite{BS02,LRS}.
These formulas express the pullback of a Schubert class as a sum of
Schubert classes on the smaller flag manifold whose coefficients are naturally Schubert
structure constants.
This was applied in~\cite{BSY} to show that quiver coefficients~\cite{Bu01,Bu05,KMS,Mi}
are naturally Schubert structure constants, as the decomposition
formula~\cite{BKTY1,BKTY2} is a special case of the formulas in~\cite{BS98,LRS}.

We generalize the formulas in~\cite{BS98,LRS}.
Let $L$ be a Levi subgroup of a semisimple algebraic group $G$ and write their flag manifolds as $\calF_L$ and 
$\calF_G$, respectively.
For each right coset of the Weyl group of $L$ in the Weyl group of $G$ there is a 
natural embedding of $\calF_L$ into $\calF_G$.
If $L$ is the Levi of a standard parabolic subgroup and $\varsigma$ is the minimal element in a 
coset, then the corresponding map on cohomology is expressed in terms of polynomial representatives as the map
on generators induced by $\varsigma$.
Analyzing the pushforward map on Schubert cycles in homology gives an expression for
the pullback map as a sum of Schubert classes for $\calF_L$ whose coefficients are naturally Schubert
structure constants for $\calF_G$.

We also give a similar formula for the pullback in $K$-theory.

In Section~\ref{S:flag}, we give background information on the cohomology and Grothendieck rings of flag
manifolds. 
Our main results are given in Section~\ref{S:pattern}, where we recall the results of Billey and Braden, and
apply them to obtain our formulas.

%
\section{Flag manifolds}\label{S:flag}
We work over the complex numbers, but our results are valid for any
algebraically closed field $k$, when we replace cohomology by Chow groups.

Let \defcolor{$G$} be a connected and simply connected complex semisimple linear algebraic group,
\defcolor{$B$} a Borel subgroup of $G$, and \defcolor{$T$} the maximal torus contained in $B$.
The Weyl group $\defcolor{W}:=N(T)/T$ of $G$ is the quotient of the normalizer of $T$
by $T$.
Our choice of $B$ gives $W$ the structure of a Coxeter group with a preferred set of
generators and a length function, $\defcolor{\ell}\colon W\to\{0,1,2,\dotsc,\}$.
Let $\defcolor{w_o}\in W$ be the longest element.

All Borel subgroups are conjugate by elements of $G$, which identifies the set $\defcolor{\calF}$ of Borel
subgroups as the orbit $G/B$, called the \demph{flag manifold}.
A Borel subgroup $B_0$ is fixed by an element $g\in G$ if and only if $g\in B_0$.
The Weyl group embeds in $\calF$ as its set of $T$-fixed points,
$\calF^T$.
These index $B$-orbits on $G/B$, which together form the \demph{Bruhat decomposition},
 \begin{equation}\label{Eq:BruhatDecomposition}
   \calF\ =\ \bigsqcup_{w\in W} BwB/B\,.
 \end{equation}
Each orbit $BwB/B$ is isomorphic to an affine space of dimension $\ell(w)$.
An orbit is a \demph{Schubert cell, $X^\circ_w$}, and its closure is a \demph{Schubert variety},
\defcolor{$X_w$}.
Set $\defcolor{B_-}:=w_oBw_o$, which is the Borel subgroup opposite to $B$ containing $T$.
Let $\defcolor{X^w}:=w_oX_{w_ow}=\overline{B_-w}$, which is also a Schubert variety and has codimension $\ell(w)$.
The intersection $X^v\cap X_w$ is nonempty if and only if $w\geq v$ and in that case it is irreducible of dimension
$\ell(w)-\ell(v)$~\cite{Deodhar,Richardson}.
%
%

Both the integral homology groups and cohomology ring of $\calF$ are free as 
$\Z$-modules with bases given by \demph{Schubert classes} associated to
Schubert varieties, and these classes do not depend upon the choice of Borel subgroup.
For homology, let $\defcolor{[X_w]}\in H_{2\ell(w)}(\calF,\Z)$ be the fundamental cycle of the
Schubert variety.
For cohomology, let $\defcolor{\frakS_w}\in H^{2\ell(w)}(\calF,\Z)$ be the cohomology
class Poincar\'e dual to $[X^w]$. 
Then $\frakS_v\cap[X_w]=[X^v\cap X_w]$, where $\cap$ is the cap product giving the action of cohomology on
homology. 

Since the Schubert classes form a basis, there are integer \demph{Schubert structure constants}
\defcolor{$c^w_{u,v}$} for $u,v,w\in W$ defined by the identity in $H^*(\calF,\Z)$
 \begin{equation}\label{Eq:SSC}
    \frakS_u\cdot \frakS_v\ =\ \sum_w c^{w}_{u,v} \frakS_w\,.
 \end{equation}
These constants vanish unless $\ell(w)=\ell(u)+\ell(v)$ and they are nonnegative, for they count the number of
points in a triple intersection of Schubert varieties,
$gX^{u}\cap X^v\cap X_{w}$, where $g\in G$ is general.
Important for us is the duality formula.
Let $\pi\colon\calF\to\rm{pt}$ be the map to a point.
Then, if $v,w\in W$, we have
 \begin{equation}\label{Eq:duality}
    \pi_*(\frakS_v\cap [X_{w}])\ =\ \left\{\begin{array}{rcl}
       1&\ &\mbox{if } v=w\\
       0&&\mbox{otherwise}\end{array}\right.\,,
 \end{equation}
so the Schubert basis is self-dual.
Combining this with~\eqref{Eq:SSC} gives
 \begin{equation}\label{Eq:pushforward_coeff}
   c^w_{u,v}\ =\ \pi_*(\frakS_u\cdot\frakS_v\cap [X_w])\,.
 \end{equation}

Recall the projection formula.
Let $f\colon Y\to Z$ be a map of compact topological spaces and $\pi\colon Y,Z\to\rm{pt}$ maps to a point.
For $y\in H_*(Y)$ and $z\in H^*(Z)$, we have 
 \begin{equation}\label{Eq:projection}
   \pi_*(z\cap f_*(y))\ =\ \pi_*(f^*(z)\cap y)\,.
 \end{equation}

The cohomology ring of the flag manifold has a second, algebraic description.
The Weyl group acts on the dual \defcolor{$\frakh^*$} of the Lie algebra \defcolor{$\frakh$} of the torus.
Borel~\cite{Bo53} showed that the cohomology of $\calF$ with complex coefficients is naturally identified with
the quotient of the symmetric algebra \defcolor{$S_{\bullet}\frakh^*$} of $\frakh^*$ by the ideal generated by its  
non-constant $W$-invariants,
 \begin{equation}\label{Eq:coh_Presen}
   H^*(\calF,\C)\ =\ S_{\bullet}\frakh^*/\langle(S_{\bullet}\frakh^*)_+^W\rangle
   \ =\ S_{\bullet}\frakh^*\otimes_{(S_{\bullet}\frakh^*)^W} \C\,.
 \end{equation}

These two descriptions, one geometric and the other algebraic, are linked.
Chevalley~\cite{Ch91} gave a formula for the product of any Schubert class by a generating
Schubert class.
This special case of the formula~\eqref{Eq:SSC} determines it and implies
expressions for a Schubert class as a polynomial in the generating classes.
A breakthrough was made when Bernstein, Gelfand, and Gelfand~\cite{BGG73} and Demazure~\cite{De74} 
gave a computable system of polynomials $\defcolor{P_w}\in S_{\bullet}\frakh^*$ for $w\in W$ such that $P_w$
represents the Schubert class $\frakS_w$. 
While not unique, these representatives depend only upon the choice of $P_{w_o}$.

The formulas we obtain in cohomology use only basic properties of cohomology, functoriality, the geometric
Schubert basis, generation by the dual of $\frakh$, duality, and the projection formula.
Consequently, such formulas exist for more general cohomology theories.
One such example is the Grothendieck ring.

Under tensor product, the Grothendieck group of vector bundles on $\calF$ modulo short exact sequences is
a ring \defcolor{$K^0(\calF)$}.
As $\calF$ is smooth, this is isomorphic to the Grothendieck group \defcolor{$K_0(\calF)$} of 
coherent sheaves on $\calF$.
A consequence of the Bruhat decomposition~\eqref{Eq:BruhatDecomposition} is that classes of structure
sheaves of Schubert varieties form a $\Z$-basis
of $K_0(\calF)=K^0(\calF)$.
Write \defcolor{$\calG_w$} for the class $[\calO_{X^w}]$ of the structure sheaf of the Schubert
variety $X^w$.

The Grothendieck ring has a presentation similar to~\eqref{Eq:coh_Presen} for cohomology~\cite{Pi}.
Let $\defcolor{\frakh^*_\Z}:=\mbox{Hom}(T,\C^\times)$ be the character group of $T$.
The representation ring \defcolor{$R(B)$} of $B$ is isomorphic to $\Z[\frakh^*_Z]$ and the representation ring
\defcolor{$R(G)$} of $G$ is its $W$-invariants, $R(B)^W$.
There is a natural map $R(B)\to K^0(G/B)$ induced by $V\mapsto G\times_BV$, for a representation $V$ of $B$.
This induces an isomorphism
\[
   R(B)\otimes_{R(G)} \Z\ \xrightarrow{\ \sim\ }\ K^0(G/B)\,.
\]
As with cohomology, there are (non-unique) representatives of Gro\-then\-dieck classes $\calG_w$
in $K^0(G/B)$ that lie in the Laurent ring $\Z[\frakh^*_Z]$, and these depend only upon
the choice of a representative for $\calG_{w_o}$~\cite{De74}.

Brion~\cite[Lemma 2]{Br02} showed that the product of these Grothendieck classes corresponds to the
intersection of Schubert varieties,
\[
   \calG_u\cdot\calG_v\ =\ [\calO_{X^u \cap gX^v}]\,,
\]
where $g\in G$ is general.

As with cohomology, the Grothendieck ring has a pairing induced by multiplication and the map to a
point, $\pi\colon\calF\to\rm{pt}$.
For sheaves $\calE,\calE'$ on $\calF$ this pairing is
\[
   \langle [\calE]\,,\,[\calE']\rangle\ :=\ 
   \pi_*([\calE]\cdot[\calE'])\ =\ \pi_*([\calE\otimes g\calE'])\,,
\]
where $g\in G$ is general and $\pi_*$ is the derived functor of global sections,
\[
   \pi_*([\calE])\ =\ \chi(\calE)\ =\ 
   \sum_{i\geq 0} \dim H^i(\calF,\calE)\,,
\]
which is the Euler-Poincar\'e characteristic of the sheaf $\calE$.

Since $\chi(\calO_{X_w})=1$ and if $v\geq w$ and $g\in G$ is general, then 
$\chi(\calO_{X_{v}\cap gX^{w}})=1$~\cite[Section~3]{BrLa}, this Schubert basis is not self-dual.
Brion and Lakshmibai showed that the dual basis is given by the ideal sheaves of the Schubert
boundaries.
Specifically, let \defcolor{$\calI_w$} be the sheaf of $\calO_{X_{w}}$-ideals that define the complement
of the Schubert cell $X^\circ_{w}$.
Then 
 \begin{equation}\label{Eq:Groth_duality}
  \pi_*(\calG_w\cdot[\calI_v])\ =\ \left\{ \begin{array}{rcl}
    1&\ &\mbox{if\ }v=w\\ 0&&\mbox{otherwise}\end{array}\right.\ .
 \end{equation}
Allen Knutson gave an expression for these classes (which is M\"obius inversion), 
 \begin{equation}\label{Eq:ideal_sheaves}
   [\calI_w]\ =\ \sum_{v\leq w} (-1)^{\ell(w)-\ell(v)} \calG_v\,.
 \end{equation}

As the Schubert classes form a basis of $K^0(\calF)$, there are integer 
\demph{Schubert structure constants $c^w_{u,v}$} for $u,v,w\in W$ defined by the identity in $K^0(\calF)$,
\[
   \calG_u\cdot\calG_v\ =\ \sum_w c^w_{u,v}\, \calG_w\,.
\]
By duality, we have 
 \begin{equation}\label{Eq:Groth_push_coeff}
   c^w_{u,v} \ =\ \pi_*(\calG_u\cdot\calG_v\cdot [\calI_{w}])\,.
 \end{equation}
These Schubert structure constants vanish unless $\ell(w)\geq\ell(u)+\ell(v)$ and they coincide with those for
cohomology when $\ell(w)=\ell(u)+\ell(v)$ (this is why we use the same notation for both).
This is because $K_0(\calF)$ is filtered by the codimension of the support of a sheaf
with the associated graded algebra the integral cohomology ring.
Thus when $\ell(w)=\ell(u)+\ell(v)$, $c^w_{u,v}\geq 0$.
In general, these constants enjoy the following positivity~\cite{Br02}, 
 \begin{equation}\label{Eq:Groth_pos}
   (-1)^{\ell(w)-\ell(u)-\ell(v)} c^w_{u,v}\ \geq\ 0\,.
 \end{equation}

%
\section{The Pattern Map}\label{S:pattern}

Let us recall the geometric pattern map and its main properties as developed by Billey and Braden~\cite{BB}.
Let $\defcolor{\eta}\colon\C^*\to T$ be a cocharacter with image the subgroup \defcolor{$T_\eta$} of $T$.
Springer~~\cite[Theorem 6.4.7]{Sp} showed that the centralizer $\defcolor{G'}:=Z_G(T_\eta)$ of $T_\eta$ in $G$
is a connected, reductive subgroup and $T$ is a maximal torus of $G'$. 
Furthermore, if $B_0\in\calF$ is a fixed point of $T_\eta$, so that $T_\eta\subset B_0$, then 
$B_0\cap G'$ is a Borel subgroup of $G'$.

If $\Blue{\calF'}:=G'/B'$ is the flag variety of $G'$, and \defcolor{$\calF^{T_\eta}$} the set of $T_\eta$-fixed
points of $\calF$, then this association $\calF^{T_\eta}\ni B_0\mapsto B_0\cap G'\in\calF'$ defines a 
$G'$-equivariant map $\defcolor{\psi}\colon\calF^{T_\eta}\to\calF'$.
Restricting to $T$-fixed points, this gives a map $\psi\colon W\to W'$, where \defcolor{$W'$} is the Weyl group of
$G'$. 
This is the Billey-Postnikov pattern map, generalizing maps on the symmetric groups coming from permutation
patterns. 
Specifically, $\psi\colon W\to W'$ is the unique map that is
(1) $W'$-equivariant in that $\psi(wx)=w\psi(x)$ for $w\in W'$ and $x\in W$, and (2) respects the
Bruhat order in that if $\psi(x)\leq\psi(wx)$ in $W'$ with $w\in W'$ and $x\in W$, then $x\leq wx$ in $W$.
Billey and Braden use this to deduce that the map $\psi$ is an isomorphism on each connected
component of $\calF^{T_\eta}$, and the connected components of $\calF^{T_\eta}$ are in bijection with right
cosets  $W'\backslash W$ of $W'$ in $W$.

Observe that $B_-\cap G'=\defcolor{B_-'}$, which is the Borel group opposite to $B'$ containing $T$.
Let \defcolor{$\calF^{T_\eta}_\varsigma$} be the component of $\calF^{T_\eta}$ corresponding to a coset
$W'\varsigma$ with $\varsigma\in W'\varsigma$ having minimal length, and let
$\defcolor{\iota_\varsigma}\colon\calF'\xrightarrow{\sim}\calF^{T_\eta}_\varsigma$ be the corresponding
section of the pattern map. 
Note that 
\[
   \calF^{T_\eta}_\varsigma\ =\ G'\varsigma\ =\ \overline{B_-'\varsigma}\ 
   \subset\ \overline{B_-\varsigma}\ =\ X^\varsigma\,.
\]
Billey and Braden also note that if $w\in W'$, then 
$\iota_\varsigma(X'_w)=X_{w\varsigma}\cap \calF^{T_\eta}_\varsigma = (X_{w\varsigma})^{T_\eta}$.
Combining these facts gives the following key lemma.

\begin{lemma}\label{L:image}
 Let $W'\varsigma$ be a coset of\/ $W'$ in $W$ with $\varsigma$ of minimal length in $W'\varsigma$ and 
 $\iota_\varsigma\colon \calF'\to\calF^{T_\eta}$ the corresponding section of the pattern map.
 Then, for $w\in W'$, we have 
\[
   \iota_\varsigma(X'_w) \ \subset\ X_{w\varsigma}\cap X^\varsigma\,.
\]
\end{lemma}

%
\subsection{The pattern map on cohomology}\label{SS:Coh}

The group $G'$ centralizing $T_\eta$ in $G$ is a Levi subgroup of some parabolic subgroup of $G$.
All parabolic subgroups of $G$ are conjugate to a \demph{standard} parabolic subgroup, which is a parabolic
subgroup containing $B$.
The set of standard parabolics is in bijection with subsets $I$ of the Dynkin diagram of $G$.

We will assume that $G'$ is the Levi subgroup of a standard parabolic corresponding to a subset $I$, and henceforth
write \defcolor{$G_I$} for $G'$ and \defcolor{$B_I$} for $B'$.
Write \defcolor{$\calF_I$} for its flag variety, which is a product of flag varieties whose factors correspond
to the connected components of $I$ in the Dynkin diagram of $G$. 
Its Weyl group is the parabolic subgroup $W_I$ of $W$, which is the subgroup generated by the simple
reflections corresponding to $I$.

The right cosets $W_I\backslash W$ are indexed by minimal length coset representatives \defcolor{$W^I$}.
Useful for us is the following proposition.

\begin{proposition}[{\protect\cite[Prop.~2.4.4]{BjBr}}]\label{P:coset}
 Let $\varsigma\in W^I$ be a minimal length representative of a coset of $W_I$ in $W$.
 For $w\in W_I$, we have $\ell(w\varsigma)=\ell(w)+\ell(\varsigma)$ and 
 the intervals $[e,w]$ in $W_I$ and $[\varsigma,w\varsigma]$ in $W$ are isomorphic.
\end{proposition}

We use this to refine Lemma~\ref{L:image}.

\begin{theorem}\label{Th:image}
  Let $\varsigma\in W^I$ be a minimal length coset representative with $\iota_\varsigma\colon\calF_I\to\calF$
  the corresponding section of the pattern map.
  Then  
\[
   \iota_{\varsigma}(X_w)\ =\ X_{w\varsigma}\cap X^\varsigma\,.
\]
\end{theorem}

\begin{proof}
 By Lemma~\ref{L:image}, we have the inclusion $\subset$.
 The result follows as both sides are irreducible of dimension $\ell(w)=\ell(w\varsigma)-\ell(\varsigma)$. 
\end{proof}

\begin{corollary}\label{C:pushforward}
 Let $\iota_{\varsigma,*}\colon H_*(\calF_I)\to H_*(\calF)$ be the map on homology induced by
 $\iota_\varsigma$.
 Then 
\[
   \iota_{\varsigma,*}[X_w]\ =\ 
   [X_{w\varsigma}\cap X^\varsigma]\ =\ \frakS_\varsigma\cap[X_{w\varsigma}]\,.
\]
\end{corollary}

We use this to compute the map $\iota_\varsigma^*$ on the Schubert basis of cohomology.

\begin{theorem}\label{Th:coh_Schub}
 Let $\varsigma\in W^I$ be a minimal length representative of a right coset of $W_I$ and
 $\iota_\varsigma\colon\calF_I\to\calF$ be the corresponding section of the pattern map with 
 $\iota_\varsigma(wB_I)=w\varsigma B$.
 Then 
\[
   \iota^*_\varsigma(\frakS_u)\ =\ 
    \sum_{w\in W_I} c^{w\varsigma}_{u,\varsigma}\, \frakS_w\,,
\]
 where $\iota_\varsigma^*\colon H^*(\calF)\to H^*(\calF_I)$ is the induced map on cohomology.
\end{theorem}

\begin{proof}
 Write \defcolor{$\iota$} for $\iota_\varsigma$ and let $u\in W$.
 Since Schubert classes form a basis of cohomology, there are integer 
 \demph{decomposition coefficients $d^w_u$} for $w\in W_I$ defined by the identity
\[
    \iota^*(\frakS_u)\ =\ \sum_{w\in W_I} d^w_u\, \frakS_w\,.
\]
 Using duality and applying the pushforward map, we have
 \begin{eqnarray*}
  d^w_u &=&  \pi_*(\iota^*(\frakS_u) \cap [X_w])\\
       &=&  \pi_*( \frakS_u \cap \iota_*[X_w])
       \ =\ \pi_*(\frakS_u\cdot \frakS_\varsigma \cap [X_{w\varsigma}]) 
       \ =\ c^{w\varsigma}_{u,\varsigma}\,,
 \end{eqnarray*}
 with the last line following from Corollary~\ref{C:pushforward} and~\eqref{Eq:pushforward_coeff}.
\end{proof}

Recall that we have
\[
  H^*(\calF,\C)\ =\ S_{\bullet}\frakh^*\otimes_{(S_{\bullet}\frakh^*)^W} \C
   \qquad\mbox{and}\qquad
  H^*(\calF_I,\C)\ =\ S_{\bullet}\frakh^*\otimes_{(S_{\bullet}\frakh^*)^{W_I}} \C\ ,
\]
as $G$ and $G_I$ have the same maximal torus.
The map $\iota_\varsigma^*$ is induced by the natural map on the symmetric algebra
$S_{\bullet}\frakh^*$ coming from the map $\varsigma\colon\frakh\to\frakh$ given by the action of $W$ on 
$\frakh$.
We deduce the following result.

\begin{theorem}\label{Th:coh_gens}
  The map $\iota_\varsigma^*$ on cohomology is induced by the map 
  $S_\bullet\varsigma\colon S_{\bullet}\frakh^*\to S_{\bullet}\frakh^*$.
  That is, for $x\in\frakh$ and $f\in S_{\bullet}\frakh^*$, this map is
\[
   (\iota_\varsigma^* f)(x)\ =\ f(\varsigma x)\,.
\]
\end{theorem}

We combine Theorems~\ref{Th:coh_Schub} and~\ref{Th:coh_gens} to get an algebraic formula for 
specializations of representatives of Schubert classes given by a minimal length coset representative $\varsigma$.

\begin{corollary}\label{C:Coh_ultimate_formula}
 Let $P_u\in S_{\bullet}\frakh^*$ be a representative of the Schubert class $\frakS_u\in H^*(\calF)$.
 Then, for $x\in\frakh$ and $\varsigma\in W^I$ a minimal length coset representative, we have
\[
   P_u(\varsigma x)\ \equiv\ 
    \sum_{w\in W_I} c^{w\varsigma}_{u,\varsigma}\, P_w(x)
   \qquad \mod \langle (S_{\bullet}\frakh^*)^{W_I}_+\rangle\ ,
\]
 where $P_w\in S_{\bullet}\frakh^*$ for $w\in W_I$ are representatives of Schubert classes in $H^*(\calF_I)$. 
\end{corollary}

\begin{remark}\label{R:product}
 The formula for $\iota_{\varsigma}^*(\frakS_u)$ in Theorem~\ref{Th:coh_Schub} gives an algorithm to compute it.
 First expand $\frakS_u\cdot\frakS_\varsigma$ in the Schubert basis of $H^*(\calF)$.
 Restrict the sum to terms of the form $\frakS_{w\varsigma}$ with $w\in W_I$, and 
 then replace $\frakS_{w\varsigma}$ by $\frakS_{w}$ to obtain the expression for $\iota_{\varsigma}^*(\frakS_u)$.
\end{remark}

\begin{example}\label{Ex:typeC}
 Suppose that $G$ is the symplectic group of Lie type $C_4$.
 Let $I$ be the subset of $C_4$ obtained by removing the long root 
 so that $G_I$ is the special linear group ${\it SL}_4$ of Lie type $A_3$.
 The Weyl group $C_4$ is the group of signed permutations whose elements are
 words $a_1\,a_2\,a_3\,a_4$, where the absolute values $|a_i|$ are distinct, and 
 the identity element is $1\,2\,3\,4$.
 The length of such a word is 
\[
   \ell(a_1\,a_2\,a_3\,a_4)\ =\ \#\{i<j\mid a_i>a_j\}\ +\ \sum_{a_i<0} |a_i|\,.
\]
If we use $\overline{a}$ to represent $-a$, then we have
\[
  \ell(3\,\overline{1}\,4\,2)\ =\ 4\,,\quad
  \ell(\overline{2}\,\overline{3}\,4\,1)\ =\ 7\,,\quad
  \mbox{and}\quad  \ell(\overline{2}\,\overline{1}\,3\,4)\ =\ 3\,.
\]

The action of $S_4$ on such words is to permute the absolute values without changing the signs.
The right cosets correspond to subsets $P$ of $\{1,\dotsc,4\}$ where the elements in that
coset take negative values. 
Here are the minimal length coset representatives
\[
   \overline{2}\,\overline{1}\,3\,4\,,\ 
   3\, \overline{2}\,4\,\overline{1}\,,\ 
   2\,3\,\overline{1}\,4\,,\  \mbox{ and }\ 
   \overline{3}\,4\,\overline{2}\,\overline{1}
\]
that correspond to subsets $\{1,2\}$, $\{2,4\}$, $\{3\}$, and $\{1,3,4\}$, respectively.

Write \defcolor{$\frakC_u$} for $u\in C_4$ for Schubert classes in this type $C$ flag manifold $\calF$ and 
$\frakS_w$ for $w\in S_4$ for Schubert classes in the type $A$ flag manifold $\calF_I$.
We will let $\varsigma=\overline{2}\,\overline{1}\,3\,4$ and compute
$\iota^*_\varsigma(\frakC_{3\,\overline{1}\,4\,2})$.
Following Remark~\ref{R:product}, we first compute
$\frakC_{3\,\overline{1}\,4\,2}\cdot\frakC_{\overline{2}\,\overline{1}\,3\,4}$.
 
We use the Pieri formula for the symplectic flag manifold as given in~\cite{BS_Lagr_P}, for 
\[
   \frakC_\varsigma\ =\ \frakC_{\overline{2}\,\overline{1}\,3\,4}\ =\ 
    \frakC_{\overline{2}\,1\,3\,4}\cdot\frakC_{\overline{1}\,2\,3\,4}\ -\ 
     2\cdot \frakC_{\overline{3}\,1\,2\,4}\,,
\]
and the Pieri formula is for multiplication by 
$\frakC_{\overline{1}\,2\,3\,4}$, $\frakC_{\overline{2}\,1\,3\,4}$, and $\frakC_{\overline{3}\,1\,2\,4}$.
We obtain
%
%
 \[
  \qquad \frakC_{3\,\overline{1}\,4\,2}\cdot\frakC_{\overline{2}\,\overline{1}\,3\,4}
    \ =\     \frakC_{\overline{3}\,\overline{2}\,4\,\overline{1}}
     \ +\   2\frakC_{2\,\overline{3}\,4\,\overline{1}}  
     \;\ +\;\   2\frakC_{\overline{4}\,\overline{3}\,1\,2}
     \ +\   2\frakC_{\overline{2}\,\overline{3}\,4\,1}
     \ +\   2\frakC_{\overline{1}\,\overline{4}\,3\,2}
     \ +\   2\frakC_{\overline{4}\,\overline{2}\,3\,1}\,.\qquad
 \]
 As only the indices of the last four terms have the form $w\varsigma$, we see that
\[
  \iota_{\varsigma}^*\bigl(\frakC_{3\,\overline{1}\,4\,2}\bigr)\ =\ 
            2\frakS_{3412}
     \ +\   2\frakS_{3241}
     \ +\   2\frakS_{4132}
     \ +\   2\frakS_{2431}\,.
\]  
\end{example}

\begin{remark}\label{R:typeA}
 The results in this section generalize results in~\cite{BS98}, which was concerned with the flag
 variety $\calF_{n+m}$ of the general linear group  ${\it GL}_{n+m}$ with root system $A_{n+m-1}$. 
 Section~4.5 of~\cite{BS98} studied an embedding of flag manifolds $\psi_P\colon\calF_n\times\calF_m\to\calF_{n+m}$
 corresponding to a subset $P$ of $\defcolor{[m{+}n]}:=\{1,\dotsc,n{+}m\}$ of cardinality $n$.
 Writing $P$ and its complement $\defcolor{P^c}:=[m{+}n]\smallsetminus P$ in order,
\[
  P\ \colon\ p_1<\dotsb<p_n
  \qquad\mbox{and}\qquad
  P^c\ \colon\ p^c_1<\dotsb<p^c_m\,,
\]
 the pullback map on cohomology
\[
   \psi_P^*\ \colon\ H^*(\calF_{n+m})\ \longrightarrow\ 
   H^*(\calF_n\times\calF_m)\ \simeq\ H^*(\calF_n)\otimes H^*(\calF_m)\,,
\]
 is induced by the map 
\[
   \psi_P^*\ \colon\ x_a\ \longmapsto \left\{ \begin{array}{rcl} 
     y_i&\ &\mbox{if } a=p_i\\
     z_j&&\mbox{if } a=p^c_j
   \end{array}\right. .
\]
 where $x_1,\dotsc,x_{n+m}$ generate $H^*(\calF_{n+m})$, 
 $y_1,\dotsc,y_n$ generate $H^*(\calF_n)$, and $z_1,\dotsc,z_m$ generate $H^*(\calF_m)$.
 The effect of $\psi^*_P$ on the Schubert basis was expressed in terms of 
 Schubert structure constants for $\calF_{n+m}$, detailed in Theorem~4.5.4 and Remark~4.5.5 of~\cite{BS98}.

 These formulas are the specialization of Theorem~\ref{Th:coh_Schub} and Corollary~\ref{C:Coh_ultimate_formula} to
 the situation of~\cite[\S~4.5]{BS98}.
 In our notation, $I$ is the subset of the Dynkin diagram $A_{n+m-1}$ obtained by removing the $n$th node,
 $G_I={\it GL}_n\times{\it GL}_m$,  $W_I=S_n\times S_m$, and 
 $\calF_I\simeq \calF_n\times\calF_m$.
 The minimal coset representative $W^I$ corresponding to the map $\psi_P$ is the inverse shuffle 
 $\varsigma_P$ defined by 
\[
   \varsigma_P\ \colon\ \left\{ \begin{array}{rclcl} 
     p_i& \mapsto& i&\ &\mbox{for }i=1,\dotsc,n\\
     p^c_j& \mapsto& m{+}j&&\mbox{for }j=1,\dotsc,m
   \end{array}\right. .
\]
This permutation is written $\varepsilon_{P,[n]}(e,e)$ in~\cite{BS98} 
and for $v\times w\in S_n\times S_m$, the permutation $(v\times w)\varsigma_P$ is
written $\varepsilon_{P,[n]}(v,w)$.

Then, in the notation we use here, Theorem~4.5.4 of~\cite{BS98} becomes Theorem~\ref{Th:coh_Schub},
\[
   \iota_{\varsigma_P}^*(\frakS_u)\ =\ 
    \sum_{v\times w\in S_n\times S_m} c^{(v\times w)\varsigma_P}_{u,\varsigma_P}\;
   \frakS_{v\times w}\,,
\]
as $\frakS_v\otimes\frakS_w=\frakS_{v\times w}$ under the K\"unneth isomorphism
$H^*(\calF_I)=H^*(\calF_n)\otimes H^*(\calF_m)$.

Finally, the map $\iota_{\varsigma_P}^*$ on $S_{\bullet}\frakh^*$ agrees with the map $\Psi_P^*$ of~\cite{BS98},
where we write the generators of $H^*(\calF_I)$ as $y_1,\dotsc,y_n, z_1,\dotsc,z_m$ as above.
\end{remark}

\subsection{The pattern map in the Grothendieck ring}\label{SS:Groth}

The results of Subsection~\ref{SS:Coh} generalize nearly immediately to Grothendieck rings of the flag varieties
$\calF$ and $\calF_I$.
In particular, Theorem~\ref{Th:image} implies the analog of Corollary~\ref{C:pushforward}.
Namely, if $w\in W_I$ and $\varsigma\in W^I$ is a minimal length coset representative, then
 \begin{equation}\label{Eq:Groth_push_sheaves}
   \iota_{\varsigma,*}(\calO_{X_w})\ =\ \calO_{X_{w\varsigma}\cap X^{\varsigma}}
     \ =\ \calO_{X_{w\varsigma}}\otimes \calO_{X^\varsigma}\,,
 \end{equation}
where $\iota_{\varsigma,*}$ is the (derived) pushforward map on sheaves, which induces the map $K_0(\calF_I)\to
K_0(\calF)$. 

Theorem~\ref{Th:coh_gens} also immediately generalizes.
The map $\iota_{\varsigma}^*\colon K^0(\calF)\to K^0(\calF_I)$ is induced by the action of 
$\varsigma$ on $\frakh_\Z^*$, leading to a formula similar to Corollary~\ref{C:Coh_ultimate_formula}
for polynomial representatives of Grothendieck classes $\calG_w$, once we generalize the
formula of Theorem~\ref{Th:coh_Schub}.

What remains is a formula for the \demph{decomposition coefficients $d^w_u$} for $u\in W$ and $w\in W_I$ defined by
the identity, 
 \begin{equation}\label{Eq:Groth_Decomp}
   \iota_\varsigma^*(\calG_u)\ =\ \sum_{w\in W_I} d^w_u \,\calG_w\,.
 \end{equation}
Using duality~\eqref{Eq:Groth_duality} for Grothendieck classes, we have
\[
   d^w_u\ =\ \pi_*(\iota_\varsigma^*(\calG_u) \cdot [\calI_w])
        \ =\ \pi_*(\calG_u\cdot \iota_{\varsigma,*}[\calI_w])\,.
\]
We prove the following lemma, which will enable this calculation.

\begin{lemma}\label{L:Dual_push}
  With these definitions, we have 
  $\iota_{\varsigma,*}(\calI_w)\ =\ \calI_{w\varsigma}\otimes \calO_{X^\varsigma}$.
\end{lemma}

As the projection formula~\eqref{Eq:projection} also holds for the Grothendieck ring/group, 
the same arguments as in the proof of Theorem~\ref{Th:coh_Schub} yield the following theorem.

\begin{theorem}\label{Th:Groth_Schub}
 With these definitions, we have
\[
   \iota^*_\varsigma(\calG_u)\ =\ 
    \sum_{w\in W_I} c^{w\varsigma}_{u,\varsigma}\, \calG_w\,,
\]
 where $\iota_\varsigma^*\colon K^0(\calF)\to K^0(\calF_I)$ is the induced map on Grothendieck rings.
\end{theorem}

\begin{proof}[Proof of Lemma~$\ref{L:Dual_push}$]
 We use the expression~\eqref{Eq:ideal_sheaves} for the ideal sheaves, the pushforward
 formula~\eqref{Eq:Groth_push_sheaves}, and Proposition~\ref{P:coset} to compute
 \begin{eqnarray*}
   \iota_{\varsigma,*}(\calI_w) &=& \sum_{v\leq w}(-1)^{\ell(w)-\ell(v)} \iota_{\varsigma,*}(\calO_{X_v})
  \ =\ \sum_{v\leq w}(-1)^{\ell(w)-\ell(v)} \calO_{X_{v\varsigma}}\otimes \calO_{X^\varsigma}\\
  &\stackrel{!}{=}& \sum_{u\leq w\varsigma} (-1)^{\ell(w)-\ell(v)} \calO_{X_{u}}\otimes \calO_{X^\varsigma}
  \ =\ \calI_{w\varsigma}\otimes \calO_{X^{\varsigma}}\,.
 \end{eqnarray*}
 The equality $(\stackrel{!}{=})$ follows as we have $u\leq w\varsigma$ and $\calO_{X_{u}}\otimes
 \calO_{X^\varsigma}$ is the zero sheaf unless $\varsigma\leq u$.
 Thus the sum over $u\leq w\varsigma$ is equal to the sum over $u$ in the interval
 $[\varsigma,w\varsigma]_\leq$ in the Bruhat order, and this interval is parameterized by 
 the interval $[e,w]$ in the Bruhat order on $W_I$ under the map $v\mapsto v\varsigma$, by
 Proposition~\ref{P:coset}. 
\end{proof}

The decomposition coefficients $d^w_u$ of~\eqref{Eq:Groth_Decomp} are nonnegative in the same sense as the
Grothendieck structure constants $c^w_{u,v}$~\eqref{Eq:Groth_pos}.
Indeed,
\[
  (-1)^{\ell(w)-\ell(u)}d^w_u\ =\ 
  (-1)^{\ell(w)-\ell(u)}c^{w\varsigma}_{u,\varsigma}\ =\ 
  (-1)^{\ell(w\varsigma)-\ell(u)-\ell(\varsigma)}c^{w\varsigma}_{u,\varsigma}\ >\ 0\,,
\]
as $\ell(w)-\ell(u)=\ell(w)+\ell(\varsigma)-\ell(u)-\ell(\varsigma)\ =\ 
   \ell(w\varsigma)-\ell(u)-\ell(\varsigma)$ by Proposition~\ref{P:coset}.

\begin{remark}\label{R:Groth_typeA}
 With the same conventions as Remark~\ref{R:typeA} the results here for the map $\iota_{\varsigma}^*$ on
 Grothendieck rings specialize to those of~\cite[Sec.\ 7]{LRS} in type $A$.
\end{remark}

\begin{remark}
 The results here should hold for more general cohomology theories, such as $T$-equivariant $K$-theory.
 We plan to treat that in a sequel.
\end{remark}
\providecommand{\bysame}{\leavevmode\hbox to3em{\hrulefill}\thinspace}
\providecommand{\MR}{\relax\ifhmode\unskip\space\fi MR }
\providecommand{\MRhref}[2]{%
  \href{http://www.ams.org/mathscinet-getitem?mr=#1}{#2}
}
\providecommand{\href}[2]{#2}

\end{document}